\documentclass[leqno,12pt]{article}
\usepackage{amsmath, amssymb}
\usepackage{amsthm}
\theoremstyle{plain} %text of this environment is typesetted in italics
\newtheorem{theorem}{\indent\sc Theorem}[section]
\newtheorem{lemma}[theorem]{\indent\sc Lemma}
\newtheorem{corollary}[theorem]{\indent\sc Corollary}

\theoremstyle{definition} %text of this environment is typesetted in roman letters
\newtheorem{definition}[theorem]{\indent\sc Definition}
\newtheorem{remark}[theorem]{\indent\sc Remark}
\newtheorem{example}[theorem]{\indent\sc Example}

\makeatletter
\def\address#1#2{\begingroup
\noindent\parbox[t]{7.8cm}{%
\small{\scshape\ignorespaces#1}\par\vskip1ex
\noindent\small{\itshape E-mail address}%
\/: #2\par\vskip4ex}\hfill%
\endgroup}%
\makeatother
\title{{$HS_{\lowercase {r}}$-valued Gauss maps and umbilic spacelike
surfaces  of codimension two
}} %title of the paper
\author{
\bigskip \\
\text{Dang Van Cuong and Doan The Hieu}\footnote{ The authors is supported in part by the National Foundation for Science and Technology Development, Vietnam (Grant No. 101.01.30.09).}}
\begin{document}

\maketitle

%%%%%%%%%%%%%%% footnote %%%%%%%%%%%%%%%%
\footnote{ %2000 MSC numbers
2000 \textit{Mathematics Subject Classification}.
Primary 00; Secondary 00.
}
\footnote{ %key words and phrases
\textit{Key words and phrases}.
Lorentz-Minkowski space, $\textbf n_r^{\pm}$-Gauss map, Umbilicity.
}

%%%%%%%%%%%%%%%%%%%%%%%%%%%%%%%%%%%%%%%%%

\begin{abstract}
To study  spacelike surfaces of codimension two in the Lorentz-Minkowski space $\Bbb R^{n+1}_1,$ we construct a pair of maps whose values are in $HS_r:=H_+^n(\textbf v,1)\cap \{x_{n+1}=r\},$  called $\textbf n_r^{\pm}$-Gauss maps. It is showed that they are well-defined and useful to study practically  flat as well as umbilic spacelike surfaces of codimension two in $\Bbb R^{n+1}_1.$

\end{abstract}
%============================================================================
\section{Introduction} %delete * to number this section

In classical differential geometry, the Gauss map plays an important role in the study of the behaviour  or geometric invariants of surfaces of codimension one. In the case of surfaces of codimension larger than one, Gauss map associated with some arbitrary normal field $\nu$ is considered. By that way, one can consider the second fundamental form associated with $\nu$ and study invariants or properties of surfaces, concerning to the concept of $\nu$-curvatures, that are dependent or independent on $\nu.$

In 1989, Marek Kossowski \cite{ko} used Gauss maps, whose values are in the lightcone, to study spacelike 2-surfaces in $\Bbb R^4_1,$ followed by Izumiya et. al. (see \cite{izu2}).
In 2004, Izumiya et. al.  \cite{izu1} used Gauss maps associated with a normal field $\nu$ to study $\nu$-umbilicity for spacelike surfaces of codimension two in Lorentz-Minkowski spaces.
Long before, in the study of minimal 2-surfaces in $\Bbb R^n,$ it is well-known that the mean curvature vector $\overrightarrow{H}$ does not depend on $\nu$ (see \cite{os}).

Motivated by these ideas, to study practically spacelike surfaces of codimension two in $\Bbb R_1^{n+1},$ we construct a kind of Gauss map whose values are in a hyperbolic space,  called $\textbf n_r^{\pm}$-Gauss maps.

 Let $M$ be a spacelike surface of codimension two in $\Bbb R^{n+1}_1.$ The normal plane of $M$ at $p\in M,$ denoted by $N_pM$  is a timelike 2-plane. We identify $N_pM$ with its image under the translation given by the vector $-p.$ Then, the intersection of $N_pM$ and  the hyperbolic space with center ${\bf v}=(0,0,\dots,0, -1)$ and radius $1$,  $H_+^n({\bf v},1)$, is a hyperbola. For a fixed $r>0$, the hyperplane $\{x_{n+1}=r\}$ meets this hyperbola exactly at two points, denoted by $\textbf n_r^{\pm}(p).$

This gives two differential maps $p\mapsto \textbf n_r^{\pm}(p),$ called $\textbf n_r^{\pm}$-Gauss maps. Their derivatives are self-adjoint, and hence we can define the $\textbf n_r^{\pm}$-Weingarten maps, $\textbf n_r^{\pm}$-Gauss-Kronecker curvatures, $\textbf n_r^{\pm}$-mean curvatures,  $\textbf n_r^{\pm}$-principal curvatures, $\textbf n_r^{\pm}$-flat points, $\textbf n_r^{\pm}$-umbilic points \ldots.

We use these maps to study the flatness and umbilicity for  spacelike surfaces of codimension
two in $\Bbb R^{n+1}_1.$

In this situation, some criteria for a spacelike surface to be flat or umbilic as well as examples of some kinds of flat and umbilic spacelike surfaces of codimension two are established. These examples show that we can use $\textbf n_r^{\pm}$-Gauss maps to study some properties of spacelike surfaces of codimension two practically.

%=============================================================

\section{Prelimineries}

\subsection{The Lorentz-Minkowski space $\Bbb R^{n+1}_1$}
 The Lorentz-Minkowski space $\Bbb R^{n+1}_1$ is the $(n+1)$-dimensional vector space $\Bbb R^{n+1}=\{( x_1, x_2,\ldots, x_{n+1}): x_i\in \Bbb R, i=1,2,\ldots, n+1\}$ endowed the pseudo scalar product
$$\langle \textbf x, \textbf y\rangle=\sum_{i=1}^{n}x_iy_i-x_{n+1}y_{n+1},$$
where $\textbf x=( x_1, x_2,\ldots, x_{n+1}), \textbf y=(y_1, y_2, \ldots y_{n+1})\in \Bbb R^{n+1}.$
Since $\langle, \rangle$ is non-positive definite, $\langle \textbf x, \textbf x\rangle$ may be zero or negative. We say a nonzero vector $\textbf x\in \Bbb R^{n+1}_1$ spacelike, lightlike or timelike if  $\langle \textbf x, \textbf x\rangle>0$, $\langle \textbf x, \textbf x\rangle=0$ or $\langle \textbf x, \textbf x\rangle<0$, respectively.
If $\langle \textbf x, \textbf y\rangle=0,$ we say $\textbf x,\textbf y$ are pseudo-orthogonal.

The norm of a vector $\textbf x\in \Bbb R^{n+1}_1$, denoted by $\|\textbf x\|$, is defined by $\sqrt{|\langle \textbf x,\textbf x\rangle|}.$
For a nonzero vector $\textbf n\in \Bbb R^{n+1}_1$, a hyperplane with the pseudo normal $\textbf n$ is defined as
$$HP(\textbf n,c)=\{\textbf x\in\Bbb R^{n+1}_1 : \langle \textbf x,\textbf n\rangle=c,\  c\in\Bbb R\}.$$
The hyperplane is said to be spacelike, lightlike or timelike if $\textbf n$ is timelike, lightlike or spacelike, respectively.

It is easy to see that,  $HP(\textbf n,c)$ is spacelike if any vector $\textbf x\in HP(\textbf n,0)$ is spacelike; $HP(\textbf n,c)$ is lightlike if  $HP(n,0)$ is tangent to the lightcone and $HP(\textbf n,c)$ is timelike if $HP(\textbf n,0)$ contains timelike vectors.

In $\Bbb R^{n+1}_1,$ we have three kinds of pseudo-hyperspheres
\begin{enumerate}
\item $H^{n}(\textbf a,R)=\{\textbf x\in\Bbb R^{n+1}_1\ |\ \langle \textbf x-\textbf a,\textbf x-\textbf a\rangle=-R^2,\ R>0\}:$ the hyperbolic with center $\textbf a$ and radius $R;$
\item $S_1^{n}(\textbf a,R)=\{\textbf x\in\Bbb R^{n+1}_1\ |\ \langle \textbf x-\textbf a,\textbf x-\textbf a\rangle=R^2,\ R>0 \}:$ the de Sitter with center $\textbf a$ and radius  $R;$
 \item $LC(\textbf a)=\{\textbf x\in\Bbb R^{n+1}_1:  \langle \textbf x-\textbf a,\textbf x-\textbf a\rangle=0\}:$  the lightcone with vertex $\textbf a.$
\end{enumerate}
And we call

 $$H_+^{n}(\textbf a,R)=\{\textbf x\in H^{n}(\textbf a,R):\ x_{n+1}-a_{n+1}\geq 0 \}$$
  the hyperbolic space with center $\textbf a$ and radius $R.$
%\item  $HP(\textbf n,c)\cap H_+^{n}(\textbf a,R)$ a hypersphere, an equidistant hypersurface or  a  hyperhorosphere if $\textbf n$ is timelike, spacelike or lightlike, respectively.

%============================================

\subsection{The $\textbf{n}_r^{\pm}$-Gauss maps}
In this paper a surface  is always spacelike and is of codimension two in $\Bbb R_1^{n+1},$ unless otherwise stated. It is an embedding $\textbf X:U\to\Bbb R^{n+1}_1$ , where $U$ is an open domain in $\Bbb R^{n-1}.$ We often  identify $M=\textbf X(U)$ with $\textbf X.$

   %In \cite{izu1}, Izumiya et. al.  study $\nu$-umbiliciy for a spacelike with respect to a given normal field $\nu.$

   In this section we introduce two concrete spacelike normal fields on a surface that are useful to study  the flatness and umbilicity, practically.

   The normal plane of $M$ at $p\in M,$ denoted by $N_pM,$  can be viewed as a timelike 2-plane passing the origin. The intersection of this plane and  the hyperbolic space with center $\text{\bf v}=(0,0,\dots,0,-1)$ and radius $1,$ $H_+^n(\text{\bf v},1)$ is a hyperbola. For a fixed $r>0$, the hyperplane $\{x_{n+1}=r\}$ meets this hyperbola exactly at two points, denoted by $\textbf {n}_r^{\pm}(p).$

 \begin{definition} The following maps
$$\begin{aligned} \textbf {n}^{\pm}_r: M&\to  HS_r:=H_+^n(\textbf v,1)\cap \{x_{n+1}=r\}\\
p&\mapsto \textbf {n}^{\pm}_r(p).
  \end{aligned} $$
  are called  $\textbf {n}^{\pm}_r$-Gauss maps.
 \end{definition}
The first property of  $\textbf {n}^{\pm}_r$-Gauss maps is
\begin{theorem}
The $\textbf{n}_r^{\pm}$-Gauss maps are smooth.
\end{theorem}
 \begin{proof}
  Locally, $\textbf {n}^{\pm}_r(p)$ are the solutions of the following system of equations
$$\begin{cases}
\langle \textbf X_{u_i},\textbf a\rangle&=0,\ \ i=1,2,\ldots, n-1;\\
 \langle \textbf a-\textbf v, \textbf a-\textbf v\rangle&=-1;\\

          \end{cases}$$
where $\textbf a=(a_1, a_2, \ldots,a_n, r).$

Since $\text{rank} (\textbf X_{u_1}, \textbf X_{u_2},\ldots, \textbf X_{u_{n-1}})=n-1,$ we can assume that  $a_1, a_2, \ldots,a_{n-1}$ are linearly expressed in term of $a_n.$ Substituting these to the last equation, we get a quadratic equation in term of $a_n.$ This equation has exactly two solutions and of course they are smooth.
\end{proof}

From now on, the symbol `` * '' means `` + '' or `` - '',
 unless otherwise stated.
%==============================================================

The derivative of  $\textbf n_{r}^{*}$ at $p$
$$d\textbf n_r^{*}(p)\ :\ T_pM\rightarrow T_{\textbf n_{r}^{*}(p)}H_+^n(\textbf v,1)\subset T_pM\oplus N_pM;$$
 can be writen as
$$d\textbf n_r^{*}(p)=d{\textbf n_r^{*}}^T(p)+d{\textbf n_r^{*}}^N(p),$$
where
$d{\textbf n_r^{*}}^T$ and $d{\textbf n_r^{*}}^N$ are the tangent and normal components of $d{\textbf n_r^{*}},$ respectively.

We recall some definitions and facts concerning to $\nu$-umbilic (see \cite{izu1}) but restated for $\textbf n_r^{*}.$
Denoted by
\begin{enumerate}
\item $A_p^{\textbf n_r^{*}}:=-d{\textbf n_r^{*}}^T(p),$  the $\textbf n_r^{*}$-Weingarten map of $M$  at $p;$
\item $K_p^{\textbf n_r^{*}}:=\det(A_p^{\textbf n_r^{*}}),$ the $\textbf n_r^{*}$-Gauss-Kronecker curvature of $M$  at $p;$
\item $H^{\textbf n_r^{*}}_p:=\frac{1}{n-1}\text{tr}(A_p^{\textbf n_r^{*}}),$  the $\textbf n_r^{*}$-mean curvature  of $M$  at $p;$
\item $k_1^{\textbf n_r^{*}}(p),k_2^{\textbf n_r^{*}}(p),\dots,k_{n-1}^{\textbf n_r^{*}}(p),$ (the eigenvalues   of $A_p^{\textbf n_r^{*}}$) the $\textbf n_r^{*}$-principal curvatures of $M$ at $p$.

\end{enumerate}
Of course
$$K_p^{\textbf n_r^{*}}=k_1^{\textbf n_r^{*}}(p)k_{2}^{\textbf n_r^{*}}(p)\dots k_{n-1}^{\textbf n_r^{*}}(p),$$
and
$$H_p^{\textbf n_r^{*}}=\frac 1{n-1}(k_1^{\textbf n_r^{*}}(p)+k_{2}^{\textbf n_r^{*}}(p)+\dots +k_{n-1}^{\textbf n_r^{*}}(p)).$$

We have some well-known facts.
\begin{enumerate}
\item  The $\text{\bf n}_r^{*}$-Weingarten map is self-adjoint.
\item  The $\text{\bf n}_r^{*}$-principal curvatures  $k_i^{\text{\bf n}_r^{*}}(p),i=1,2,\dots,n-1$ of $M$ at $p$  are the solutions of the following equation
\begin{equation} \label{principal} \det(b_{ij}^{\text{\bf n}_r^{*}}(p)-kg_{ij}(p))=0,\end{equation}
where  $b_{ij}^{\text{\bf n}_r^{*}}(p):=\langle \textbf X_{u_iu_j}(p),\text{\bf n}_r^{*}(p)\rangle,\ i,j=1,2,\dots,n-1 ,$ the coefficients of the $\text{\bf n}_r^{*}$-second fundamental form of $M$ at $p.$

\item $K_p^{\text{\bf n}_r^{*}}={\det(b_{ij}^{\text{\bf n}_r^{*}}(p))}.{\det(g_{ij}(p))}^{-1}.$
\end{enumerate}

%=================================================
\begin{definition}

\begin{enumerate}
\item A point $p\in M$ is said to be $\textbf n^{*}_r$-umbilic if $k_i^{\textbf n^{*}_r}(p)=k(p),\ i=1,2,\ldots, n-1.$ If $k(p)=0,$ then $p$ is called  $\textbf n^{*}_r$-flat.
\item $M$ is said to be $\textbf n^{*}_r$-umbilic ($\textbf n^{*}_r$-flat) if every point $p\in M$ is $\textbf n^{*}_r$-umbilic ($\textbf n^{*}_r$-flat).
\item $M$ is said to be totally umbilic (totally flat) if every point $p\in M$ is $\textbf n^{*}_{r}$-umbilic ($\textbf n^{*}_{r}$-flat) for every $r>0.$
    \end{enumerate}
\end{definition}

%========================================================
\section{The $\textbf n_r^{*}$- flatness}
We begin with a useful lemma.
\begin{lemma} \label{lem1} If $(\text{\bf n}_r^*)_{u_i}\in N_pM,$ where $i\in\{1,2,\ldots, n-1\},$  then $(\text{\bf n}_r^*)_{u_i}=0.$
\end{lemma}
\begin{proof} We observe that, the last coordinate of $(\textbf n_r^*)_{u_i}$ is zero because the last coordinate of $\textbf n_r^*$ is constant. Therefore, since $\{\textbf n_r^+, \textbf n_r^-\}$ is a basis of $N_pM,$ we have
\begin{equation}\label{pt1}(\textbf n_r^*)_{u_i}=\lambda (\textbf n_r^+- \textbf n_r^-).\end{equation}

An easy calculation shows that $\langle \textbf n_r^*,\textbf n_r^*\rangle=2r.$ Therefore,

$$\langle (\textbf n_r^*)_{u_i},\textbf n_r^*\rangle=\lambda\langle \textbf n_r^+-\textbf n_r^-,\textbf n_r^*\rangle=0.$$

 If $\lambda\ne 0,$ then
 $$\langle \textbf n_r^+,\textbf n_r^+\rangle=\langle \textbf n_r^-,\textbf n_r^-\rangle=\langle \textbf n_r^+,\textbf n_r^-\rangle=2r. $$
And hence,
$$ \langle \textbf n_r^+-\textbf n_r^-,\textbf n_r^+-\textbf n_r^-\rangle=0,$$
a contradiction, because $\textbf n_r^+-\textbf n_r^-$ is a nonzero spacelike vector. Thus, $\lambda= 0,$ and the lemma is proved.
\end{proof}

%========================================

\begin{theorem}\label{theoflat1}Let $M$ be a connected surface. The following statements are equivalent
\begin{enumerate}
\item there exists an $r>0,$\ $M$ is $\text{\bf n}_r^*$-flat;
\item there exists an $r>0,$\ $\text{\bf n}_r^*$  is constant;
\item there exists a spacelike vector $\text{\bf a}=(a_1,a_2,\dots,a_n, a_{n+1}), a_{n+1}\ne 0$ and a real number $c$ such that $M\subset HP(\text{\bf a},c).$
    \end{enumerate}
\end{theorem}
\begin{proof}

($1.\Rightarrow 2.$)\ Since  $M$ is $\textbf n_r^*$-flat, i.e. $A_p^{\textbf n_r^*}=0,$ we have
\begin{equation}\label{eq1}\langle \textbf X_{u_iu_j},\textbf n_r^*\rangle=-\langle \textbf X_{u_i},(\textbf n_r^*)_{u_j}\rangle=0,\ i,j=1,2,\dots,n-1.\end{equation}

But (\ref{eq1}) means that $(\textbf n_r^*)_{u_i}\in N_pM$ and hence $(\textbf n_r^*)_{u_i}=0, \ i=1,2,\dots,n-1$ by virtue of Lemma \ref{lem1}.

($2\Rightarrow 1$)\ Obviously.

($2.\Rightarrow 3.$)\
If $\textbf n_r^*$ is constant, then
$$\frac{\partial}{\partial u_i}\langle\textbf X,\textbf n_r^*\rangle=\langle\textbf X_{u_i},\textbf n_r^*\rangle-\langle\textbf X,(\textbf n_r^*)_{u_i}\rangle=0.$$
Thus $\textbf X\subset H(\textbf n_r^*, c),$ for some constant $c.$

($3.\Rightarrow 2.$)
If $M$ is contained in a timelike hyperplane with a unit spacelike normal vector $\textbf a=(a_1, a_2,\dots, a_n,a_{n+1}), a_{n+1}\ne 0$, then it is not hard to check  that we can choose the constant vector $\textbf n_r^*=2a_{n+1}\textbf a\in H^n_+(\text{\bf v},1).$
\end{proof}

\begin{remark}
 \begin{enumerate}
\item The Theorem \ref{theoflat1} is a necessary and sufficient condition for a surface belonging to a timelike hyperplane that does not contain the $x_{n+1}$-axis. For the case of surfaces belonging to a timelike hypersurface containing the $x_{n+1}$-axis, see Example \ref{exflat}.
\item A necessary and sufficient condition for a surface belonging to a lightlike hyperplane based on the totally lightlike flatness was established in \cite{izu3}.
\item A similar result with an assumption on parallelism of the normal field was given (\cite[Theorem 4.3]{izu1}).
    \end{enumerate}
\end{remark}

\begin{corollary}
Let $M$ be a connected surface and $\text{\bf n}_{r_1}^*\ne\text{\bf n}_{r_2}^*,$  i.e.   $r_1\ne r_2$ or $\text{\bf n}_{r_1}^*= \text{\bf n}_{r}^+,\ \ \text{\bf n}_{r_2}^*=\text{\bf n}_{r}^-$ for some fixed $r.$ If $M$ is both $\text{\bf n}_{r_1}^*$- and $\text{\bf n}_{r_2}^*$-flat, then $M$ is a part of a spacelike $(n-1)$-plane. In this cases, $\text{\bf n}_r^{*}$ are constant  for every $r>0$ or equivalently, $M$ is totally flat, i.e. $\text{\bf n}_r^{*}$-flat  for every $r>0.$
\end{corollary}
\begin{corollary}
If $M$ is connected and contained in a timelike hyperplane not containing the $x_{n+1}$-axis, then there exists a unique possitive real number $r$ such that $M$ is $\text{\bf n}_r^*$-flat unless $M$ is (or a part of) a spacelike $(n-1)$-plane.
\end{corollary}

%===========================

\section{The $\text{\bf n}_r^*$-umbilicity}

In this section, we study the $\text{\bf n}_r^{*}$-umbilicity for  spacelike surfaces of codimension two in $\Bbb R_1^{n+1}.$ For a pseudo-hypersphere, we mean a hyperbolic or a de Sitter with center $\textbf a$ and radius $R,$ or a lightcone with vertex $\textbf a.$ Because $\textbf n_{r}^{\pm}$-umbilicity is an invariant under translations, we can assume that $\textbf a$ is the origin in the study of the $\text{\bf n}_r^{*}$-umbilicity for surfaces lying in a pseudo-hypersphere. We begin this section with another useful lemma.

\begin{lemma}\label{funlem}
Suppose that $\nu_1$ and $\nu_2$ are smooth normal fields on $M$ and for every $p\in M, \ \nu_1(p), \nu_2(p)$ are linear independent. If $M$ is both $\nu_1$- and $\nu_2$-umbilic then $M$  is $\nu$-umbilic for every smooth normal field $\nu.$
\end{lemma}
\begin{proof}
By the assumption, for every smooth normal field $\nu$
$$\nu=\lambda_1\nu_1+\lambda_2\nu_2$$
where $\lambda_i,\ i=1,2$ are smooth functions on $M.$

Because $d(\lambda_i\nu_i)^T=\lambda_i d(\nu_i)^T,\ i=1,2$
$$ A^{\nu}=\lambda_1 A^{\nu_1}+\lambda_2 A^{\nu_2}.$$
Since $A^{\nu_i}=k^{\nu_i}\text{id},\ i=1,2$
$$ A^{\nu}=(\lambda_1 k^{\nu_1}+\lambda_2 k^{\nu_2})\text{id}.$$
\end{proof}

 Because $\textbf n_{r}^+,$  $ \textbf n_{r}^-$ are linear independent by the construction and so are $\textbf n_{r_1}^*,\ \textbf n_{r_2}^*$ if $r_1\ne r_2,$  we have
\begin{corollary}
If $M$ is $\textbf n_{r_1}^*$- and $\textbf n_{r_2}^*$-umbilic, where $\textbf n_{r_1}^*\ne\textbf n_{r_2}^*;$ then $M$ is totally umbilic.
\end{corollary}

\begin{remark}
\begin{enumerate}
\item  By virue of Lemma \ref{funlem}, a surface is totally umbilic iff it is $\nu$-umbilic for every smooth normal field $\nu.$
\item It is well-known that (see \cite[Lemma 4.1]{izu1}), a surface lying in a pseudo-hypersphere is always  $\nu$-umbilic, where $\nu$ is the position vector field. Therefore, Lemma \ref{funlem} is useful in the  study of the totally umbilicity for surfaces lying in a hyperbolic or a lightcone, because the position vector field and $\textbf n_{r}^*$ are always linear independent. The case of the de Sitter can be studied in a similar way by using the lightcone Gauss maps (see \cite{izu2}, \cite{ko}...). So for simplicity in statements, we just state  for the case of the hyperbolic spaces.
    \end{enumerate}
\end{remark}
By using of Theorem \ref{theoflat1}, Lemma \ref{funlem} or by a direct computation (see Example  \ref{exflat}), we have
\begin{corollary}\label{cor54}
If $M$ is contained in the intersection of a hyperbolic space and a hyperplane, then $M$ is totally umbilic.
\end{corollary}

\begin{theorem}\label{theoum1}
Let $M$ be a spacelike  surface of codimension two in $H_+^n(0,R).$ The following statements are equivalent.
\begin{enumerate}
\item  there exists $r>0,$\ $M$ is $\textbf n_r^*$-umbilic;
\item   $M$ is totally umbilic;
\item $M$ is contained in a hyperplane.
    \end{enumerate}
\end{theorem}

\begin{proof}

($1. \Rightarrow 2.$)\ Because $M$ is contained in $H_+^n(0,R),$ \ $M$ is umbilic with respect to the position vector field $\textbf X.$ Moreover, because $\textbf X$ is timelike  while $\textbf n_r^*$ is spacelike, $M$ is totally umbilic by virtue of Lemma \ref{funlem}.

($2. \Rightarrow 3.$) Let $$\nu=\frac{\textbf X\wedge \textbf X_{u_1}\wedge \textbf X_{u_2}\wedge \dots \wedge \textbf X_{u_{n-1}}}{\left|\textbf X\wedge \textbf X_{u_1}\wedge \textbf X_{u_2}\wedge \dots \wedge \textbf X_{u_{n-1}}\right|}.$$
Because
\begin{equation}\label{hapdan}\langle \nu,\textbf X\rangle =0,\qquad \langle \nu,\nu\rangle =\pm 1,\qquad  \langle \nu,\textbf X_{u_i}\rangle =0,\ i=1,2,...,n-1;\end{equation}
we have
$$\langle d\nu,\textbf X\rangle = \langle \nu,d\textbf X\rangle =0;\  \langle \nu,d\nu\rangle =0.  $$
Since $\{\nu,\textbf X\}$ is a basis of $N_pM,$  $d\nu\in T_pM,$ i.e. $\nu$ is parallel.

By virtue of Lemma 4.2 in \cite{izu1}, $d\nu=\lambda d\textbf X,$ where $\lambda$ is constant and hence $\nu=\lambda \textbf X+ \textbf a,$ where
$\textbf a$ is a constant vector.
 Since
$\langle \nu,\textbf X\rangle =0, $ \  $\langle \textbf X, \textbf a\rangle=-\langle \textbf X,\lambda \textbf X\rangle=-\lambda R=c$\ (a constant).
Thus, $M\subset HP(\textbf a,c)$.\\
($3.\Rightarrow 1.$) follows by Corollary \ref{cor54}.
\end{proof}

%================================================
\begin{lemma}\label{lempara}
Let $\nu_1,\ \nu_2$ be parallel vector fields on the connected surface $M$ and $\nu=\alpha\nu_1+\beta\nu_2.$ Suppose that for every $p\in M,$ $\nu_1(p), \nu_2(p)$ are linear independent,  then $\nu$ is parallel if and only if $\alpha$ and $\beta$ are constants.
\end{lemma}

\begin{proof}
The assumption that $\nu, \ \nu_1,\ \nu_2$ are parallel yields
$$d\alpha\nu_1+d\beta\nu_2=0.$$
But this implies $d\alpha=d\beta=0$ since   $\nu_1, \nu_2$ are linear independent.
Conversely, it is obvious that if  $\alpha$ and $\beta$ are constants then $\nu$ is parallel.

\end{proof}

Among all hyperspheres $HP(\textbf n,c)\cap H_+^{n}(0,R)$ \ ($\textbf n$ is timelike) of the hyperbolic  space $H_+^{n}(0,R),$ the case of right hyperspheres, i.e. $\textbf n=(0,0,\ldots, 1),$ are special. The following theorem give some necessary and sufficient conditions for a  surface lying in a hyperbolic space to be a part of a right hypersphere.

%==============================

\begin{theorem}\label{theoum2} Let $M$ be a  surface  contained in  $H_+^{n}(0,R).$ The following statements are equivalent:
\begin{enumerate}
\item $M$ is contained in a right hypersphere;
\item $\text{\bf n}_r^{*}$ is parallel for any $r>0;$
\item there exists two different parallel normal fields $\text{\bf n}_{r_1}^*,\ \text{\bf n}_{r_2}^*;$
\item there exists $r>0,$ such that $A^{\text{\bf n}_r^*}=-\alpha{\text{id}},$ where $\alpha$ is constant.
\end{enumerate}
\end{theorem}
%{\bf Note: Not hold for de Sitter, because $X$ is spacelike. Maybe just state for the Hyperbolic and the lightcone}
\begin{proof}
($1. \Rightarrow 2.$)\ It is not hard to see that, because $M\subset\{x_{n+1}=c\}\cap H_+^{n}(0,R),$  for every $r>0,$
\begin{equation}\label{chinhhang}\textbf n_r^*=\alpha\textbf X+ \beta\textbf v,\end{equation}
where $\alpha,\beta$ are constants. Since $\textbf X$ is parallel and $\textbf v=(0,0,\ldots,0, -1)$ is constant, $\textbf n_r^*$ is parallel.

%It is not hard to verify that the first statement implies the rest ones.

($2. \Rightarrow 3.$) Obviously.

($3. \Rightarrow 1.$)\
Because $\textbf X$ is a parallel normal field and $\{\text{\bf n}_{r_1}^{*},\text{\bf n}_{r_2}^{*}\}$ is a basis of $N_pM,$ we have the linear expression
\begin{equation}\label{bieuthi}\textbf X=\alpha \textbf n_{r_1}^*+\beta \textbf n_{r_2}^*,\end{equation}
where $\alpha,\beta$ are constants by virtue of Lemma \ref{lempara}. Since the last coordinates of  $\text{\bf n}_{r_1}^{*}$ and $\text{\bf n}_{r_2}^{*}$ are constants, the last coordinate of $\textbf X$ is constant.

($1.\Rightarrow 4.$) The equation (\ref{chinhhang}) implies that
$$A^{\textbf n_r^*}=-\alpha\text{id}.$$

($4.\Rightarrow 1.$) By the assumption, $M$ is $\textbf n_r^*$-umbilic. By virtue of Theorem \ref{theoum1}, $M\subset HP(\textbf a,c),$ where $\textbf a$ is a unit vector. Except at most one point, where $\textbf X$ is parallel to $\textbf a,$
$$\textbf n_r^*=\alpha\textbf X+\beta\textbf a,$$
where $\beta$ is a differential function on $M$.

Since $\langle\textbf n_r^*,\textbf n_r^*\rangle=2r,\ \ \langle\textbf X,\textbf X\rangle=-R^2,  \ \ \langle\textbf X,\textbf a\rangle=c$ we obtain the following equation
$$2r=-\alpha^2 R^2+2\alpha c\beta +\textbf a^2\beta^2 .$$
Thus, $\beta$ is constant and therefore so is the last coordinate of $\textbf X.$
\end{proof}
%=================================
 The following theorem give another necessary and sufficient condition for a surface to be a part of a right hypersphere of a hyperbolic space without the assumption  of lying in the hyperbolic space.

\begin{theorem}\label{hang} Let $M$ be a  surface in $\Bbb R_1^{n+1}.$ The following statements are equivalent
\begin{enumerate}
\item there exists $r>0$ such that $\textbf n_r^*$ is parallel, not constant,  and $M$ is $\textbf n_r^*$-umbilic;
\item $M$ is contained in a right hypersphere in a hyperbolic space.
\end{enumerate}
\end{theorem}
\begin{proof}
($(1)\Rightarrow (2)$)\  Since $M$ is $\textbf n_r^* $-umbilic,\ $\textbf  n_r^*$ is parallel and $\langle \textbf  n_r^*,\textbf  n_r^*\rangle =2r;$ $d\textbf  n_r^*=\alpha d\textbf X,\ \alpha=\text{const.}\ne 0,$ by virtue of Lemma  4.2 in \cite{izu1}. Therefore,
\begin{equation}\label{111} \textbf  n_r^*=\alpha\textbf X+\textbf a,\end{equation}
where $\textbf a$ is constant.

Let $\textbf v=(0,0,\dots,0,-1).$ From (\ref{111}) we have
$$\textbf X-\frac{1}{\alpha}(\textbf v-\textbf a)=\frac{1}{\alpha}(\textbf  n_r^*-\textbf v).$$
A simple calculation yields
 $$\langle \text X-\frac{1}{\alpha}(\textbf v-\textbf a),\langle \text X-\frac{1}{\alpha}(\textbf v-\textbf a)\rangle=-\frac{1}{\alpha^2},      $$
i.e. $M$ is contained in the hyperbolic space with center $\frac{1}{\alpha}(\textbf v-\textbf a) $ and radius $R=\frac{1}{\alpha},$ and hence contained in a right hypersphere by virtue of Theorem \ref{theoum2}.

($(2)\Rightarrow (1)$)\ is obvious by Theorem \ref{theoum2}.
\end{proof}
%===========================================================
The following is somewhat similar to the first statement of Lemma 4.2 in \cite{izu1}.
\begin{theorem} \label{theoum4}\ Let  $M$ be a connected surface in $\Bbb R_1^{n+1}.$ If there exists $r>0,$ such that $M$ is $n_r^*$-umbilic and for every $i, j \in\{1,2,\ldots, n-1\}$
\begin{equation}\label{dk}[(\textbf n_r^*)^T_{u_i}]_{u_j}=[(\textbf n_r^*)^T_{u_j}]_{u_i}\end{equation}
 then $A_p^{\textbf n_r^*}=-\alpha\text{id},$ where $\alpha$ is constant.
\end{theorem}
\begin{proof}
By the assumption, we have
$$ (\textbf n_r^+)_{u_i}^T=\alpha \textbf X_{u_i},\ \ i=1,2,\ldots n-1.$$
Therefore, for every $i, j \in\{1,2,\ldots, n-1\}$
$$[(\textbf n_r^*)^T_{u_i}]_{u_j}=\alpha_{u_j}\textbf X_{u_i}+\alpha \textbf X_{u_iu_j},$$
and
$$[(\textbf n_r^*)^T_{u_j}]_{u_i}=\alpha_{u_i}\textbf X_{u_j}+\alpha \textbf X_{u_ju_i}.$$
Since $[(\textbf n_r^*)^T_{u_i}]_{u_j}=[(\textbf n_r^*)^T_{u_j}]_{u_i}$ and $\textbf X_{u_iu_j}=\textbf X_{u_ju_i},$ we have
$$\alpha_{u_i}\textbf X_{u_j}-\alpha_{u_j}\textbf X_{u_i}=0;$$
and hence  $\alpha_{u_i}=\alpha_{u_j}=0$ because  $\textbf X_{u_i},\textbf X_{u_j}$ are linear independent;  and therefore $\alpha$ is constant because $M$ is  connected.
\end{proof}

%================== EXAMPLES
\section{Examples}
We construct some concrete examples to illustrate the above results.

\begin{example} This example shows that there exists an $\textbf n_r^*$-umbilic surface but not totally umbilic.

  Let $M$ be a parametric surface in $\mathbb \Bbb R^4_1$, defined by  the parametric equation
$$\textbf X(u,v)=\left(\frac{1}{2}u^2,au-\frac{1}{2}u^2,u^2+v^2,u\right),\ v>0,\ u>1,\ a=\sqrt 3-1$$

A direct computation shows that $\textbf X$ is spacelike and
$$\textbf n_a^-=\left(1,1,0, a\right),$$
$$\textbf n_a^+=\left(\frac{-a^2+4ua-2u^2}{a^2-2ua+2u^2},\frac{a^2-2u^2}{a^2-2ua+2u^2},0,a  \right).$$

Since  $\textbf n_a^-$ is constant, $M$ is $\textbf n_a^-$-flat. We can check that $\textbf X\subset HP(\textbf n_a^-, 0).$

Calculating the first and the second fundamental forms (with respect to $\textbf n_a^+$) yields
$$(g_{ij})=\left(\begin{matrix}6u^2-2au+a^2-1&4uv\\4uv&4v^2 \end{matrix}\right),$$
and
$$(b_{ij}(\textbf n_a^+))=\begin{pmatrix}\frac{-2a^2+4au}{a^2-2au+2u^2}&0\\0&0 \end{pmatrix}.$$
Therefore, the principal curvatures $k_1^{\textbf n_a^-}$ and $k_2^{\textbf n_a^-}$ are the solutions of the following equation
$$\label{kho}4v^2\left(2u^2-2au+a^2-1\right)k^2-4v^2\left(\frac{-2a^2+4au}{a^2-2au+2u^2} \right)k  =0.$$
It is easy to see that $k_1^{\textbf n_a^+}=0$ and $k_2^{\textbf n_a^+}\ne 0.$ Thus, $M$ is not $\textbf n_a^+$-umbilic.

\end{example}
%==========================================

\begin{example}\label{exflat}\ This is an  example of a totally umbilic surfaces, but the curvature $\lambda$ is not constant.
 %{\bf Dang kiem tra . Totally umbilic, not parallel, ki not constant, containing in SP chua {xn+1}}

Consider the equidistance hypersurface in $H_+^3(0,1)$
$$M=H_+^3(0,1)\cap \{x_1=0\}=\textbf X(\mathbb R^2)$$
defined by
$$\textbf X(u,v)=(0,u,v,\sqrt{u^2+v^2+1});\ (u,v)\in\mathbb R^2.$$

A direct computation yields
$$\textbf X_u=\left(0,1,0,\frac{u}{\sqrt{u^2+v^2+1}}\right),\ \textbf X_v=\left(0,0,1,\frac{v}{\sqrt{u^2+v^2+1}}\right);$$
$$g_{11}=\frac{v^2+1}{u^2+v^2+1},\ \ \ g_{12}=g_{21}=\frac{-uv}{u^2+v^2+1},\ \ \ g_{22}=\frac{u^2+1}{u^2+v^2+1};$$

$$\textbf n_r^{\pm}=\left(\pm\sqrt{\frac{r^2}{u^2+v^2+1}+2r},\frac{ru}{\sqrt{u^2+v^2+1}},\frac{rv}{\sqrt{u^2+v^2+1}},r\right);$$

$$\textbf X_{uu}=\left(0,0,0,\frac{v^2+1}{(u^2+v^2+1)^{3/2}}\right),\textbf X_{vv}
=\left(0,0,0,\frac{u^2+1}{(u^2+v^2+1)^{3/2}}\right),$$
$$\textbf X_{uv}=\left(0,0,0,\frac{-uv}{(u^2+v^2+1)^{3/2}}\right);$$

$$(g_{ij})=\frac{1}{u^2+v^2+1}\left(\begin{matrix}v^2+1&-uv\\-uv&u^2+1\end{matrix}\right);\ (g_{ij})^{-1}=\left(\begin{matrix}u^2+1&uv\\uv&v^2+1\end{matrix}\right);$$

$$ (b_{ij}^{\textbf n_r^{\pm}})=\frac{-r}{(u^2+v^2+1)^{3/2}}\left(\begin{matrix}v^2+1&-uv\\-uv&u^2+1\end{matrix}\right);$$

$$(a_{ij}^{\textbf n_r^{\pm}})=(b_{ij}^{\textbf n_r^{\pm}})(g_{ij})^{-1}=\frac{-r}{\sqrt{u^2+v^2+1}}\left(\begin{matrix}1&0\\0&1\end{matrix}\right);$$

$$[(\textbf n_r^+)^T_{u}]_{v}=\left(0,\frac{-rv}{\sqrt{(u^2+v^2+1)^3}},0,\frac{-2ruv}{(u^2+v^2+1)^2}  \right);$$
$$[(\textbf n_r^+)^T_{v}]_{u}=\left(0,0,\frac{-ru}{\sqrt{(u^2+v^2+1)^3}},\frac{-2ruv}{(u^2+v^2+1)^2}\right).$$

We can see that $M$ is totally umbilic. Moreover, $$k_p^{\textbf n_r^{\pm}}=\frac{-r}{\sqrt{u^2+v^2+1}}$$
is not constant and $[(\textbf n_r^+)^T_{u}]_{v}\ne [(\textbf n_r^+)^T_{v}]_{u}$ (see Theorem \ref{theoum4}).
\end{example}

%=======================================================================
\begin{example}%{\bf Da kiem tra}
This is an example of a $\nu$-umbilic but neither $\textbf n_r^+$- nor $\textbf n_r^-$-umbilic for any $r\in \Bbb R_+.$
Let $$\textbf X:(0,\frac{\pi}{2})\times(-\frac{\pi}{2},0)\rightarrow \mathbb R_1^4,\ \ \ \ (u,v)\mapsto (u,\sin v,v,\cos u).$$
A direct computation yields
$$\textbf X_u=(1,0,0,-\sin u),\ \ \ \textbf X_v=(0,\cos v,1,0);$$
$$ \textbf X_{uu}=(0,0,0,-\cos u),\ \ \textbf X_{uv}=\textbf X_{vu}=(0,0,0,0),\ \ \textbf X_{vv}=(0,-\sin v,0,0);$$
$$g_{11}=\langle \textbf X_u,\textbf X_u\rangle=\cos^2u>0,\ \ g_{12}=\langle \textbf X_u,\textbf X_v\rangle=0,\ \ g_{22}=\langle \textbf X_v,\textbf X_v\rangle=1+\cos^2v>0;$$

$$\textbf n_r^+=\left(-r\sin u,-\sqrt{\frac{r^2\cos^2u+2r}{1+\cos^2v}},\cos v\sqrt{\frac{r^2\cos^2u+2r}{1+\cos^2v}},r\right);$$
$$\textbf n_r^-=\left(-r\sin u,\sqrt{\frac{r^2\cos^2u+2r}{1+\cos^2v}},-\cos v\sqrt{\frac{r^2\cos^2u+2r}{1+\cos^2v}},r\right);$$

$$(b_{ij}^{\textbf n_r^{\pm}})=\begin{pmatrix}r\cos u&0\\0&\mp\sin v\sqrt{\frac{r^2\cos^2u+2r}{1+\cos^2v}}\end{pmatrix}; $$

$$(g_{ij})=\begin{pmatrix}\cos^2u&0\\0&1+\cos^2v\end{pmatrix}; $$

\begin{equation}\label{mtA}  (a_{ij}^{\textbf n_r^{\pm}})
=(b_{ij}^{\textbf n_r^{\pm}}).(g_{ij})^{-1}
=\begin{pmatrix}\frac{r}{\cos u}&0\\0&\mp\sin v\sqrt{\frac{r^2\cos^2u+2r}{(1+\cos^2v)^3}}\end{pmatrix};    \end{equation}

 \begin{equation}\label{dcc}  k_1^{\textbf n_r^{\pm}}(P)=\frac{r}{\cos u},\ k_2^{\textbf n_r^{\pm}}(p)=\mp\sin v\sqrt{\frac{r^2\cos^2u+2r}{(1+\cos^2v)^3}}.    \end{equation}

 At each point $p=x(u,v)\in M,$ let $\nu(p)= \text n_{r_{p}},$ where $r_{p}=\frac{2\sin^2v\cos^2 u}{(1+\cos^2v)^3-\cos^4u\sin^2v}.$ We can see that $\nu$ is a smooth normal vector field on $M$ and $M$ is $\nu$-umbilic but neither $\textbf n_r^+$- nor $\textbf n_r^-$-umbilic for any $r\in \Bbb R_+.$
\end{example}
%=====================================================
\begin{example}\ %{\bf Da kiem tra}\
 Let
$$\textbf X(u,v)=\left(u,\sin v,\cos v,\sqrt{2+u^2}\right),\quad u\in\mathbb R,\ \  v\in(-\pi/2,\pi/2). $$
Because $\langle \textbf X,\textbf X\rangle =-1,$ $M\subset H_+^4(0,1).$ A direct computation yields
$$\textbf X_u=\left(1,0,0,\frac{u}{\sqrt{u^2+2}}\right) ,\qquad \textbf X_v=\left(0,\cos v,-\sin v,0\right);  $$
$$g_{11}=\frac{2}{2+u^2},\ g_{12}=g_{21}=0,\ g_{22}=1;$$
$$\textbf n_r^{\pm}=\left(\frac{ru}{\sqrt{u^2+2}},\pm\sin v\sqrt{-\frac{u^2r^2}{u^2+2}+r^2+2r  },\pm\cos v\sqrt{-\frac{u^2r^2}{u^2+2}+r^2+2r  },r \right);$$
$$\textbf X_{uu}=\left(0,0,0,\frac{2}{(u^2+2)^{3/2}} \right),\ \ \textbf X_{uv}= \textbf X_{vu}= \left(0,0,0,0\right),\ \ \textbf X_{vv}=\left(0,-\sin v,-\cos v,0\right)$$

$$b_{11}^{\textbf n_r^{\pm}}=\frac{-2r}{(u^2+2)^{3/2}},\ b_{12}^{\textbf n_r^{\pm}}=0,\ b_{22}^{\textbf n_r^{\pm}}=\mp\sqrt{\frac{2r(u^2+r+2)}{u^2+2} }; $$

 $$k_1^{\textbf n_r^{\pm}}=\frac{-r}{\sqrt{u^2+2}},\ \  k_2^{\textbf n_r^{\pm}}=\mp\sqrt{\frac{2r(u^2+r+2)}{u^2+2}}.$$

We can see that $k_1^{\textbf n_r^+}> k_2^{\textbf n_r^+}$ while $k_1^{\textbf n_r^-}< k_2^{\textbf n_r^-}$ for any $r>0.$ Thus, $M$ is not $\nu$-umbilic, for any normal vector field $\nu\ne\textbf X.$

\end{example}

%%%%%%%%%%%%%%%%%%%%%%%%%%%%%%%%%%%%%%%%%%%%

\bigskip
%%%%%%%%%%%% Authors' addresses %%%%%%%%%%%%%
\address{ Dang Van Cuong\\
Hue Geometry Group\\
Departement of Mathematics \\
Duy Tan University \\
Danang \\
Vietnam
}
{cuongdangvan@gmail.com}
\address{Doan The Hieu\\
Hue Geometry Group\\
Departement of Mathematics \\
College of Education, Hue University \\
Hue \\
Vietnam
}
{dthehieu@yahoo.com}

\end{document}